\documentclass{amsart}
\usepackage{amssymb}
\usepackage{graphicx}
\usepackage{enumerate}
\usepackage{multirow}
\usepackage{amsmath,color}
\usepackage{hyperref}
\usepackage{url}

\newtheorem{thm}{Theorem}[section]
\newtheorem{cor}[thm]{Corollary}
\newtheorem{lem}[thm]{Lemma}
\newtheorem{prop}[thm]{Proposition}

\numberwithin{equation}{section}

\DeclareMathOperator{\diam}{diam}

\newcommand{\eva}{\theta}
\newcommand{\evac}{\overline{\theta}}

\title{Trees with a large Laplacian~\linebreak eigenvalue multiplicity}

\author{S. Akbari}
\address{Department of
Mathematical Sciences, Sharif University of Technology, Iran}
\email{s\_akbari@sharif.edu}
\author{E.R. van Dam}
\address{Department of Econometrics and O.R., Tilburg University,
	 The Netherlands}
\email{Edwin.vanDam@uvt.nl}
\author{M.H. Fakharan}
\address{Department of
Mathematical Sciences, Sharif University of Technology, Iran}
\email{mh.fakharan92@student.sharif.ir}
\date{}
\begin{document}

\subjclass[2010]{05C50}

\keywords{Laplacian spectrum, trees, multiplicities of eigenvalues}

\begin{abstract}
In this paper, we study the multiplicity of the Laplacian eigenvalues of trees.
It is known that for trees, integer Laplacian eigenvalues larger than $1$ are simple and also the multiplicity of Laplacian eigenvalue $1$ has been well studied before.
Here we consider the multiplicities of the other (non-integral) Laplacian eigenvalues. We give an upper bound  and determine the trees of order $n$ that have a multiplicity that is close to the upper bound $\frac{n-3}{2}$, and emphasize the particular role of the algebraic connectivity.
\end{abstract}

\maketitle

\section{Introduction}

In this paper, we study the multiplicity of the Laplacian eigenvalues of trees. In particular, we are interested in upper bounds for these multiplicities, and trees with multiplicities that are close to these upper bounds.

Let $G$ be a graph with Laplacian eigenvalues $0=\mu_1\leq \cdots \leq \mu_n$. Fiedler \cite{fiedler} showed that $m_G(\mu_1)=1$ if and only if $G$ is connected, where $m_G(\mu)$ denotes the multiplicity of Laplacian eigenvalue $\mu$. Grone, Merris, and Sunder \cite[Prop.~2.2]{kamel} proved that $m_G(\mu_n)=1$ if $G$ is a connected bipartite graph (in particular, if it is a tree). On the other hand, for the complete graph $K_n$, the multiplicity of $\mu_n$ is large and indeed is $n-1$. How large can the multiplicities be when we restrict to trees? It is shown by Grone et al.~\cite[Thm.~2.1]{kamel} that if an integer $\mu \geq 2$ is a Laplacian eigenvalue of a tree $T$ of order $n$, then $m_T(\mu)=1$ and $\mu \mid n$. Interestingly, it is different for the Laplacian eigenvalue $1$, because $m_G(1)\geq p(G)-q(G)$ for all graphs $G$, where $p(G)$ and $q(G)$ are the number of pendant vertices and quasipendant vertices, respectively, see \cite[p.~263]{far}. For example, it follows that if $G$ is a star ($p(G)=n-1,q(G)=1$), then $m_G(1)=n-2$. Guo, Feng, and Zhang \cite{feng} showed that if $T$ is a tree of order $n$, then $m_T(1) \in S=\{0,\, 1,\, \ldots,\,n-5, n-4,\, n-2\}$ and for every integer $n$ and every $k\in S$ there exists a tree $T$ of order $n$ with $m_T(1)=k$. Also Barik, Lal, and Pati \cite{barik} studied the multiplicity of Laplacian eigenvalue $1$ in trees. In this paper we consider the multiplicities of the other (non-integral) Laplacian eigenvalues.

Our paper is further organized as follows. In Section \ref{secNotation}, we introduce some notation and  definitions. We then collect some relevant results from the literature in Section \ref{secSurvey}. In Section \ref{secExamples}, we introduce the family of trees that we will call spiders, as they will play a key role in the main part of the paper, Section \ref{secMAIN}. In this main part we will give an upper bound $\frac{n-3}{2}$ on the multiplicities of a Laplacian eigenvalue of a tree of order $n$, in particular for the algebraic connectivity, and characterize the trees that are close to this upper bound.
One main result (Theorem \ref{km2}) --- and starting point of several refinements --- is that for a tree $T$ of order $n\geq 6$, the multiplicity of a Laplacian eigenvalue $\mu \neq 1$ is at most $\frac{n-3}{2}$, and we characterize the trees that attain this bound. We will also characterize the trees that have a multiplicity that is close to the upper bound (in Theorems \ref{n/2-2} and \ref{n-1/2-2}), and show that there are finitely many trees besides a specific family of spiders for which the algebraic connectivity has a multiplicity within a fixed range from the upper bound (in Theorem \ref{thm:n/2-k}).

\section{Preliminaries}
\subsection{Notation}\label{secNotation}

Let $G = (V(G),E(G))$ be a graph, where $V (G) = \{v_1, \ldots , v_n\}$ is the vertex set and $E(G)$ is the edge set. Throughout this paper all graphs are simple, that is, without loops and multiple edges. Let $d_G(v_i)$ be the degree of $v_i$. We denote the non-increasing vertex degree sequence of $G$ by $(d_1, \ldots , d_n)$. The maximum and minimum degree of $G$ are denoted by $\Delta(G)=d_1$ and $\delta(G)=d_n$, respectively.

The {\em adjacency matrix} of $G$, denoted by $A(G)$, is an $n\times n$ matrix whose $(i, j)$-entry is $1$ if $v_i$ and $v_j$ are adjacent and $0$ otherwise. We call $L(G) = D(G) - A(G)$ the {\it Laplacian matrix} of $G$, where $D(G)$ is the $n \times n$ diagonal matrix with $d_{ii}=d(v_i)$. The eigenvalues of $L(G)$ are called the {\it (Laplacian) eigenvalues}\footnote{Whenever we write about eigenvalues of a graph, we mean Laplacian eigenvalues (unless explicitly stated otherwise).} of $G$, and we denote these in increasing order by $0=\mu_1(G) \leq \cdots \leq \mu_n(G)$. The multiset of eigenvalues of $L(G)$ is called the (Laplacian) spectrum of $G$. Fiedler \cite{fiedler} called $\mu_2$ the {\it algebraic connectivity} of $G$. The multiplicity of a Laplacian eigenvalue $\mu$ in a graph $G$ is denoted by $m_G(\mu)$; the number of Laplacian eigenvalues of $G$ in an interval $I$ is denoted by $m_GI$. We say that two (or more) Laplacian eigenvalues are conjugates (of each other) if they are roots of the same irreducible factor of the characteristic polynomial of the Laplacian (over the rationals). Conjugate eigenvalues have the same multiplicity.

A {\em star} (graph) is a complete bipartite graph $K_{1,n}$, for some positive integer $n$. We let $P_n$ be the path of order $n$. The {\em diameter} of $G$ is  denoted by $\diam(G)$. A vertex of degree one is called a {\em pendant vertex} and a vertex that is adjacent to at least one pendant vertex is called a {\em quasipendant vertex}. The number of pendant and quasipendant vertices of a graph $G$ are denoted by $p(G)$ and $q(G)$, respectively.


In all of the above notation, we remove the additional $G$ or $T$ if there is no ambiguity; for example $m(\mu)$ instead of $m_T(\mu)$, or $V$ instead of $V(G)$.

\subsection{A collection of elementary results}\label{secSurvey}

In this section we collect and extend some relevant basic lemmas from the literature.
We start with some general ones. For other basic results, we refer to the books by Brouwer and Haemers \cite{BH}, Cvetkovi\'{c}, Doob, and Sachs \cite{CDS}, and Cvetkovi\'{c}, Rowlinson, and Simi\'{c} \cite{cvet}.

\begin{lem}\label{cve}
{\em \cite[Prop.~7.5.6]{cvet}.} If $G$ is a connected graph with $r$ distinct Laplacian eigenvalues, then $\diam(G)\leq r-1$.
\end{lem}

 \begin{lem}\label{haem}
{\em \cite[Thm.~1]{haemer}}. Let $G$ be a graph of order $n$, with vertex degrees $d_1\geq \cdots \geq d_n$ and Laplacian eigenvalues $\mu_1 \leq \cdots \leq \mu_n$. If $G$ is not $K_m \cup (n-m)K_1$, then $\mu_{n-m+1}\geq d_m - m+2$ for $1\leq m \leq n$. In particular, if $G$ has at least one edge, then $\mu_n\geq d_1+1$ and if $G$ has at least two edges, then $\mu_{n-1} \geq d_2$.
\end{lem}

We recall that a sequence $\mu_1 \geq \cdots \geq \mu_m$ interlaces another sequence $\lambda_1 \geq \cdots \geq \lambda_n$ with $m < n$ whenever $\lambda_i \geq \mu_i \geq \lambda_{n-m+i}$, for $i = 1, \ldots , m$. It is well known \cite[Thm.~1]{hwa} that the eigenvalues of a principal submatrix of a Hermitian matrix $M$ interlace the eigenvalues of $M$. For the Laplacian eigenvalues, we moreover have the following two specific results.

\begin{lem}\label{kame3}
{\em \cite[Thm.~4.1]{kamel}.} Let $\tilde{G}$ be a graph of order $n$ and let $G$ be a (spanning) subgraph of $\tilde{G}$ obtained by removing just one of its edges. Then the $n-1$ largest Laplacian eigenvalues of $G$ interlace the Laplacian eigenvalues of $\tilde{G}$.
\end{lem}

A consequence of this is the following.

\begin{lem}\label{kame22}
{\em \cite[Cor.~4.2]{kamel}.} Let $v$ be a pendant vertex of $\tilde{G}$ and let $G=\tilde{G}\setminus v$. Then the Laplacian eigenvalues of $G$ interlace the Laplacian eigenvalues of $\tilde{G}$.
\end{lem}

The following results concern multiplicities.


\begin{lem}\label{kame}
\cite[p.~263]{far}. Let $G$ be a graph with $p$ pendant vertices and $q$ quasipendant vertices. Then $m_G(1)\geq p-q$.
\end{lem}
This result follows from the observation that for every pair $(v_i,v_j)$ of pendant vertices that is adjacent to the same quasipendant vertex, the difference $e_i-e_j$ of the corresponding characteristic vectors is an  eigenvector of $L(G)$ for eigenvalue $1$. This gives $p-q$ linearly independent eigenvectors.

Finally, we have some specific results for trees.

\begin{lem}\label{Tm2p}
{\em \cite[Thm.~2.3]{kamel}.} Let $T$ be a tree with $p$ pendant vertices. If $\mu$ is a Laplacian eigenvalue of $T$, then $m_T(\mu) \leq p-1$.
\end{lem}

The following is a clear generalization to non-integral eigenvalues of a result by Grone, Merris, and Sunder \cite[Thm.~2.15]{kamel}. It uses the Matrix-Tree theorem, which states that any cofactor of the Laplacian matrix of $G$ equals the number of spanning trees of $G$. In case of trees, this is equivalent to the fact that the product of all non-zero Laplacian eigenvalues equals the number of vertices; we shall use this also later.

\begin{lem}\label{van}
Let $\mu$ be a Laplacian eigenvalue of a tree $T$ with $m_T(\mu)>1$. Then the product of $\mu$ and its conjugate eigenvalues equals $1$. In particular, if $\mu$ is an integer, then $\mu=1$.
\end{lem}

\begin{proof}
Let $v$ be a pendant vertex. Then the Laplacian matrix of $T$ has the form
$$L(T)=\begin{bmatrix}
B& *\\
*& 1
\end{bmatrix},$$
where $B$ is the principal submatrix corresponding to $T\setminus v$. Since $m_T(\mu)>1$, there is an eigenvector $x$ of $L(T)$ for $\mu$ such that its last component is $0$.  If $x'$ is the vector obtained from $x$ by deleting the last component, then it is not hard to see that $Bx'=\mu x'$. So $\mu$ is an eigenvalue of $B$ as well (and so are its conjugates). By the Matrix-Tree Theorem however, we have that $\det(B)=1$, and the result follows.
\end{proof}

To conclude this section, we mention a bound for the multiplicity of the algebraic connectivity of a tree. It follows easily from a result of Grone and Merris \cite{merris}.
\begin{prop}\label{delta}
Let $T$ be a tree with   $\Delta \geq 2$. Then $m_T(\mu_2)\leq \Delta -1$.
\end{prop}

\begin{proof}
Let $m=m_T(\mu_2)$, and suppose that $m\geq \Delta$. Because the multiplicity is at least $2$, there is an eigenvector that has value $0$ corresponding to  one of the vertices, and so $T$ is a so-called type I tree (see \cite{merris}). By \cite[Thm.~2]{merris}, the so-called characteristic vertex of $T$ has degree at least $m+1$, but this is at least $\Delta+1$, which is a contradiction.
\end{proof}

\section{Spiders and their spectra}\label{secExamples}

In this section we define two families of trees that are most relevant to our results. We start with the main one, i.e., the family of trees that have large multiplicities for some non-integral Laplacian eigenvalues.

The spider $T(s,k)$, with $1\leq k \leq s$, is defined as in Figure \ref{12star}. It is obtained from the star $K_{1,s}$ by extending $k$ of its rays (legs) by an extra edge, and has $n=s+k+1$ vertices. We say that the spider has $k$ legs.

\begin{figure}[h]
\centering
\includegraphics[width=60mm]{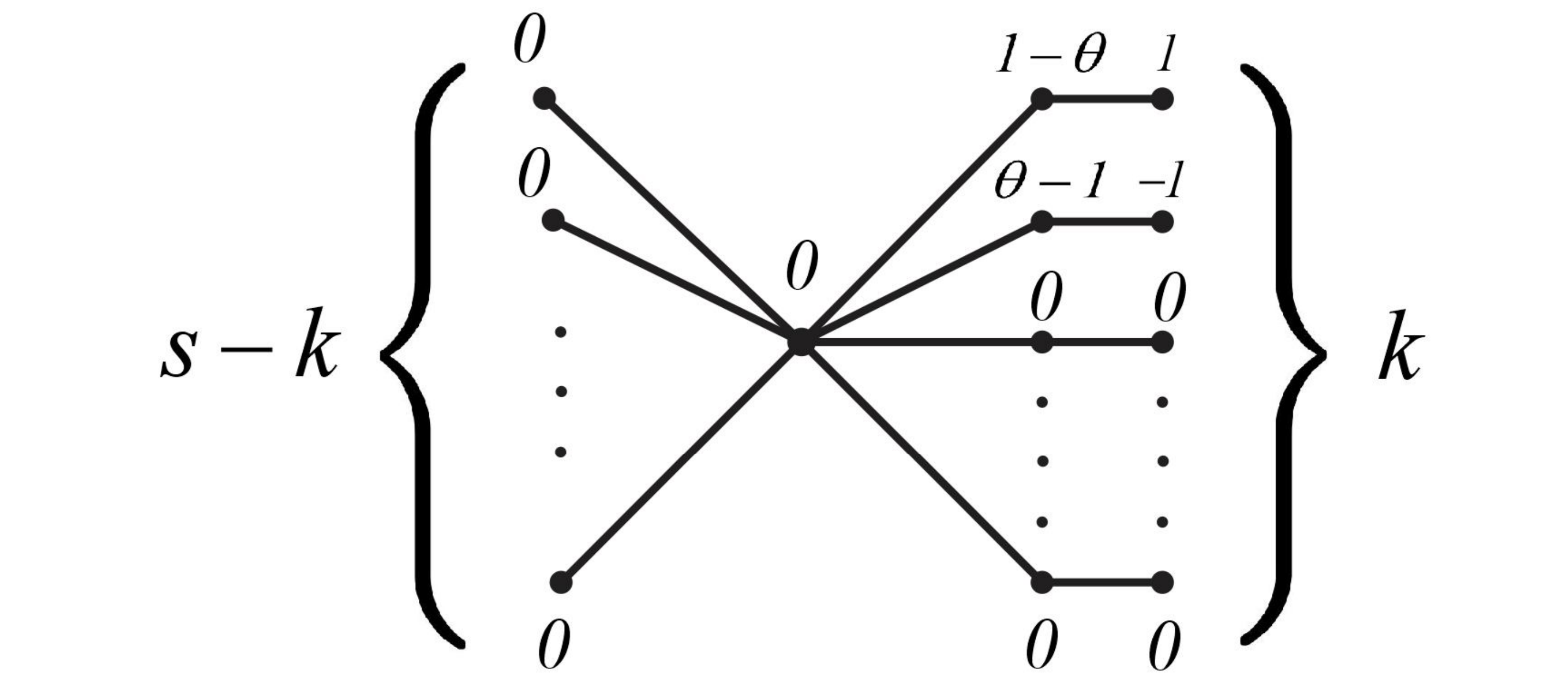}
\caption{The spider $T(s,k)$}\label{12star}
\end{figure}

\begin{prop}\label{spectrum1}
The Laplacian spectrum of $T(k,k)$ is
$$\{0^{[1]},\, \eva^{[k-1]},\, \lambda_1^{[1]},\, \evac^{[k-1]},\, \lambda_2^{[1]}\},$$
where $(\eva, \evac)=(\frac{3-\sqrt{5}}{2},\frac{3+\sqrt{5}}{2})$ and $\lambda_1$, $\lambda_2$ are the roots of  $x^2-(k+3)x+2k+1$.
\end{prop}

\begin{proof}
In Figure \ref{12star}, an eigenvector for eigenvalue $\eva=\frac{3-\sqrt{5}}{2}$ is shown\footnote{Whenever we use $\eva$ and $\evac$ in this paper, we mean the specific values $\frac{3-\sqrt{5}}{2}$ and $\frac{3+\sqrt{5}}{2}$.}. It is clear that there are $k-1$ linearly independent such eigenvectors corresponding to the eigenvalue $\eva$. Similarly, we find $k-1$ linearly independent eigenvectors corresponding to eigenvalue $\evac=\frac{3+\sqrt{5}}{2}$.
Because $\mu_1=0$, there are only two other eigenvalues $\lambda_1 \leq \lambda_2$. Because the product of all non-zero eigenvalues equals $n=2k+1$, it follows that $\lambda_1 \lambda_2 =2k+1$. Noting the trace of $L$ (which equals twice the number of edges) it follows that $\lambda_1 + \lambda_2 =k+3$, and the result follows.
\end{proof}

Note that if we let $\lambda_2 >\lambda_1$, then $\lambda_1>1$ and $\lambda_2>\evac$, so for $k \geq 2$ the algebraic connectivity $\mu_2$ equals $\frac{3-\sqrt{5}}{2}$ and $\mu_{n-1}=\frac{3+\sqrt{5}}{2}$. Proposition \ref{spectrum1} thus shows that the upper bound of Proposition \ref{delta} is sharp.

\begin{prop}\label{spectrum2}
The Laplacian spectrum of $T(s,k)$ with $k<s$ is
$$\{ 0^{[1]},\, \eva^{[k-1]},\, \lambda_1^{[1]},\, 1^{[s-k-1]},\, \lambda_2^{[1]},\, \evac^{[k-1]},\, \lambda_3^{[1]}\},$$
where $(\eva, \evac)=(\frac{3-\sqrt{5}}{2},\frac{3+\sqrt{5}}{2})$ and $\lambda_1$, $\lambda_2$, $\lambda_3$ are the roots of $x^3-(s+4)x^2+(3s+4)x-(s+k+1)$.
\end{prop}

\begin{proof} As before, we have eigenvalues $\eva$ and $\evac$ with multiplicities at least $k-1$, and eigenvalue $0$. In addition, we have eigenvalue $1$ with multiplicity at least $s-k-1$ by Lemma \ref{kame}. Thus, three eigenvalues $\lambda_1, \lambda_2,\lambda_3$ are left. Again, from the product of all non-zero eigenvalues being equal to $n$, it follows that   $\lambda_1 \lambda_2 \lambda_3 = n= s+k+1$. Moreover, $\lambda_1+\lambda_2 +\lambda_3= 2n-2-(s-k-1)-3(k-1)=s+4$. A quadratic equation easily follows from the trace of $L^2$:
$$ \sum_{i=1}^n \mu_i^2= \sum_{i=1}^n d_i^2 + 2m$$
for every graph $G$ of order $n$ with $m$ edges, and degrees $d_i$. For $T(s,k)$, we obtain that
$\lambda_1^2+\lambda_2^2 +\lambda_3^2= s^2+2s+8$, and from this we conclude that $\lambda_1 \lambda_2 + \lambda_1 \lambda_3 +\lambda_2 \lambda_3 = \frac{1}{2}(s+4)^2-\frac{1}{2}(s^2+2s+8) = 3s+4$. From all of this, the spectrum follows.
\end{proof}

Note that for a fixed integer $s$, one can use induction to show that $\nobreak \mu_2(T(s,s-i))\geq \eva$ and $\mu_{n-1}(T(s,s-i))\leq \evac$, for $i=0, \ldots , s-1$, starting from Proposition \ref{spectrum1} and by inductively applying Lemma \ref{kame22}. On the other hand, for fixed $k\geq 2$, it also follows by induction that $\mu_{n-1}(T(k+i,k))\geq \evac$ for all $i$. Consequently, we have the following.

\begin{lem}\label{lemmatsk35} Let $T=T(s,k)$, with $s \geq k \geq 2$. Then $\mu_2(T)= \eva=\frac{3-\sqrt{5}}{2}$ and $\mu_{n-1}(T)= \evac=\frac{3+\sqrt{5}}{2}$.
\end{lem}

It also follows from the above that $\mu_2(T(s,1))>\eva$ for $s\geq 1$.
Below we will show that the spider graphs $T=T(s,k)$, with $s \geq k \geq 2$ are extremal regarding the eigenvalues $\eva$ and $\evac$.

Observe first however that if $k=s-1$ then $m_T(1)=0$ and $\lambda_2=2$. So in this case, $\lambda_1$ and $\lambda_3$ are the roots of $x^2-(s+2)x+s$ and the spectrum of $T(s,s-1)$ is
$$\{0^{[1]},\, \eva^{[s-2]},\, \lambda_1^{[1]},\, 2^{[1]},\, \evac^{[s-2]},\, \lambda_3^{[1]}\}.$$

\begin{figure}[h]
\centering
\includegraphics[width=60mm]{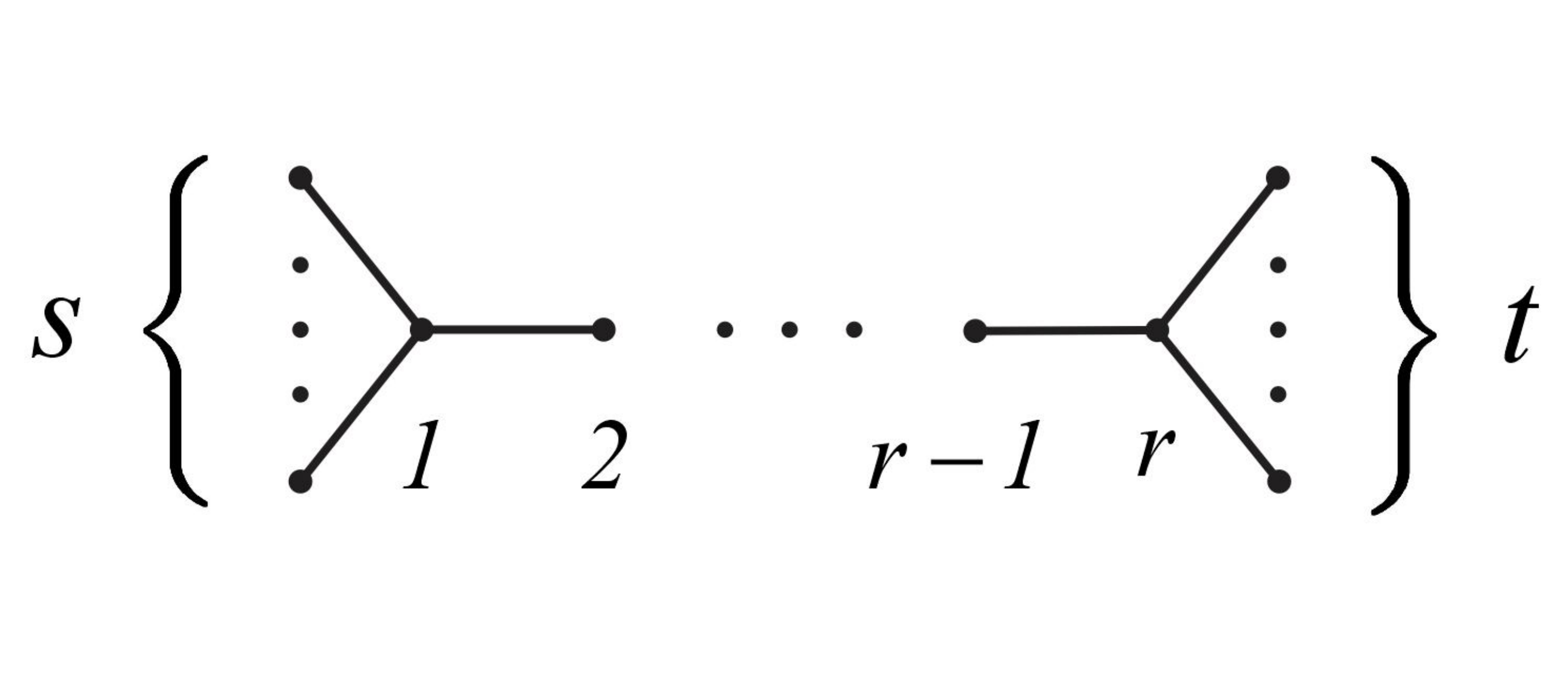}
\caption{$H(s,r,t)$}\label{Htree}
\end{figure}
A second family of trees that we will use is shown in Figure \ref{Htree}. Such a tree $H(s, r, t)$  is obtained by attaching $s$ pendant vertices to one end point of the path on $r$ vertices, and $t$ pendant vertices to its other end point, for some positive integers $s,r,t$. We note that for $r=1$, we obtain a star (a tree with diameter $2$). The Laplacian spectrum of this graph is $\{0^{[1]},\, 1^{[n-2]},\, n^{[1]}\}$. Also, every tree of order $n$ with diameter $3$ is a graph $H(s,2,t)$ for some integers $s, t$ such that $s+t=n-2$. The spectrum of this graph can easily be obtained in a similar way as in Proposition \ref{spectrum2}.

\begin{lem}\label{dou}
\cite[Prop.~1]{Grone}. Let $T=H(s,2,t)$ be of order $n=s+t+2$. Then the characteristic polynomial of $L(T)$ is
$x(x-1)^{n-4}[x^3-(n+2)x^2+(2n+st+1)x-n]$.
\end{lem}

In order to prove some of the results in the next section, we need the following two half-known characterizations of trees with extremal Laplacian eigenvalues $\frac{3 \pm \sqrt{5} }{2}$. The first is a result by Li, Guo, and Shiu \cite{li}. For completeness, we provide a shorter proof of this result.

\begin{figure}[h]
\centering
\includegraphics[width=30mm]{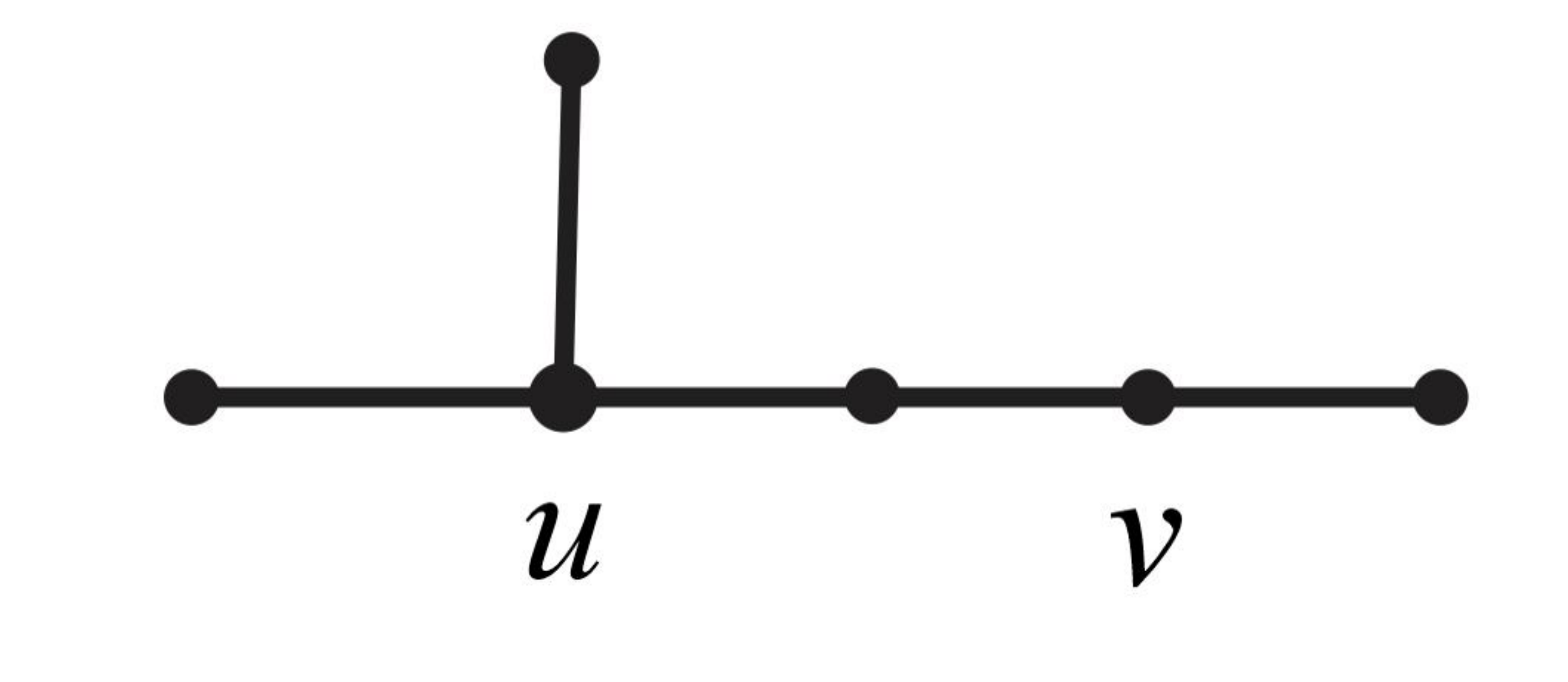}
\caption{The fork $F$}\label{obt}
\end{figure}

\begin{prop}\label{kho}
{\em \cite[Thm.~3.4]{li}.}
Let $G$ be a connected graph of order $n$. Then $\mu_{n-1}(G)=\frac{3+\sqrt{5}}{2}$ if and only if $G\cong T(s,k)$ for some positive integers $s,k$ with $k\geq 2$.
\end{prop}

\begin{proof}
One side is clear by Lemma \ref{lemmatsk35}. To show the other side, suppose that $\mu_{n-1}(G)=\frac{3+\sqrt{5}}{2}=\evac$. If $\Delta(G)=2$, then $G$ is a path or a cycle. Since $\mu_{5}(P_6)=3>\evac$, it follows from interlacing (Lemmas \ref{kame3} and \ref{kame22}) that the order of $G$ is less than $6$, and then it can be checked that $G\cong P_5 = T(2,2)$. Next, assume that $u$ is a vertex of $G$ with degree $\Delta(G)\geq 3$. If there exists a vertex $v$ at distance $2$ from $u$ with degree at least $2$, then the graph $G$ can be obtained from the fork $F$ in Figure \ref{obt} by adding some pendant vertices and some edges. But $\mu_5(F) =3>\evac$ and so again by interlacing, we have that $\mu_{n-1}(G)\geq \mu_5(F) > \evac$, which is a contradiction. Thus, every vertex of degree at least $2$ is adjacent to $u$. By Lemma \ref{haem}, we have $d_2(G)\leq \mu_{n-1}$, hence $d_2(G) \leq 2$. So $G\cong T(s,k)$ for some integers $s,k$. By Proposition \ref{spectrum2}, it is necessary to have $k\geq 2$.
\end{proof}

We note that the proof actually shows that if $\mu_{n-1}(G) \leq \frac{3+\sqrt{5}}{2}$, then $G$ is a star or a spider. For more details and the full classification of graphs with $\mu_{n-1}(G) \leq 3$, we refer to Li, Guo, and Shiu \cite{li}. The following similar result is a strengthening of a result by Zhang \cite[Thm.~2.12]{zhang}.

\begin{prop}\label{kho2}
Let $T$ be a tree of order $n$. If $\mu_2(T) \geq \frac{3-\sqrt{5}}{2}$, then $T$ is a star, a spider, $H(2,2,2)$, or $H(3,2,2)$, and equality holds if and only if $T\cong T(s,k)$ for some positive integers $s,k$, with $k\geq 2$.
\end{prop}

\begin{proof}
Suppose that $\mu_2(T) \geq \frac{3-\sqrt{5}}{2}=\eva$. By interlacing (Lemma \ref{kame22}), the tree $T$ cannot have $P_6$ as a subgraph since $\mu_2(P_6)<\eva$, so $\diam(T)\leq 4$. If $T$ is not a star, then it must have diameter $3$ or $4$. If $\diam(T)=3$, then $T$ is $H(s,2,t)$ for some positive integers $s,t$. Using Lemma \ref{dou}, we find however that $\mu_2(H(3,2,3)) < \eva $ and $\mu_2(H(4,2,2)) < \eva$, so it follows again by interlacing that in this case $T$ can only be one of the trees $H(2,2,2)$, $H(3,2,2)$, and $T(s,1)$, for some positive integer $s$. Indeed, $\mu_2(H(2,2,2)) > \eva $, $\mu_2(H(3,2,2)) > \eva$, and $\mu_2(T(s,1))> \eva$.

Finally, suppose that $\diam(T)=4$. For the fork $F$ (see Figure \ref{obt}), we have $\mu_2(F)<\eva$. So $T$ cannot have $F$ as a subgraph, and thus $T$ is a spider. The case of equality now follows from Lemma \ref{lemmatsk35} and the fact that a star, $H(2,2,2)$, $H(3,2,2)$, and $T(s,1)$ all have $\mu_2>\eva$.
\end{proof}

\section{Trees with a large multiplicity}\label{secMAIN}

We are now ready for our main results, that is, to bound the multiplicities of non-integer Laplacian eigenvalues of trees, and to characterize the trees with large Laplacian eigenvalue multiplicities.

\begin{thm}\label{km2}
Let $T$ be a tree of order $n\geq 6$ with a Laplacian eigenvalue $\mu$. If $\mu\neq 1$, then $m_T(\mu)\leq \frac{n-3}{2}$, and equality holds if and only if $T\cong T(\frac{n-1}{2},\frac{n-1}{2})$ and $\mu =\frac{3 \pm \sqrt{5}}{2}$. In particular, if equality holds, then $\mu$ is the algebraic connectivity of $T$ or its conjugate.
\end{thm}

\begin{proof}
Suppose that $m_T(\mu) > \frac{n-3}{2}$ and $\mu \neq 1$. Then $m_T(\mu)>1$. So  by Lemma \ref{van}, the eigenvalue $\mu$ is not an integer, so it has at least one conjugate eigenvalue $\overline{\mu}$ with $m_T(\overline{\mu})=m_T(\mu)$. So $m_T(\mu)\leq \frac{n-1}{2}$. If $n$ is odd, then $m_T(\mu) =\frac{n-1}{2}$, so $T$ has exactly $3$ distinct eigenvalues. This implies that the diameter of $T$ is at most $2$ (by Lemma \ref{cve}), so $T$ is a star, with spectrum $\{ 0^{[1]},\, 1^{[n-2]},\, n^{[1]}\}$, which is a contradiction.

So $n$ is even and $m_T(\mu)=\frac{n-2}{2}$. Hence $T$ has spectrum
$$ \{ 0^{[1]},\, \mu^{[\frac{n-2}{2}]},\, \overline{\mu}^{[\frac{n-2}{2}]},\, \lambda^{[1]}\}$$
for some eigenvalue $\lambda$. Now, by Lemma \ref{cve}, $T$ has diameter at most $3$, and it is not a star. So $T$ has $p=n-2$ pendant vertices and $q=2$ quasipendant vertices. Thus, $m_T(1)\geq p-q=n-4 \geq 2$, by Lemma \ref{kame}, which is a contradiction, so $m_T(\mu) \leq \frac{n-3}{2}$.

Now, suppose that $m_T(\mu)= \frac{n-3}{2}$. Again, by Lemma \ref{van}, the eigenvalue $\mu$ is not an integer.
If $n=7$, then possibly $\mu$ can have two conjugate eigenvalues. However, this would imply  that $T$ has four distinct eigenvalues, which would give a contradiction in the same way as in the above case of $n$ even. So $n \geq 7$ and $\mu$ has exactly one conjugate eigenvalue $\overline{\mu}$. By Lemma \ref{van}, we have $\mu\overline{\mu}=1$. So the spectrum of $T$ is
$$\{0^{[1]},\, \mu^{[\frac{n-3}{2}]},\, \overline{\mu}^{[\frac{n-3}{2}]},\, \lambda_1^{[1]},\, \lambda_2^{[1]}\},$$
for some eigenvalues $\lambda_1$ and $\lambda_2$. Since the product of all non-zero eigenvalues of a tree equals $n$, we find that $\lambda_1\lambda_2=n$. Because the only tree on $n$ vertices with an eigenvalue $n$ is the star (which easily follows by observing that the complement of such a tree should be disconnected), it follows that $m_T(1)=0$, and hence by Lemma \ref{kame}, we deduce that $p=q$. Moreover, by Lemma \ref{cve} the diameter of $T$ is at most $4$. Because $n$ is odd, it now easily follows that $T$ is isomorphic to $T(\frac{n-1}{2},\frac{n-1}{2})$. As observed in Proposition \ref{spectrum1}, this graph indeed has an eigenvalue (in fact, the two conjugates $\frac{3\pm \sqrt{5}}2)$ with multiplicity $\frac{n-3}{2}$.
\end{proof}

Note that if $T$ is a tree but not a star, then $\diam(T) \geq 3$. Because $\mu_2(P_4)<1$, it then follows by interlacing that $\mu_2(T)<1$, so we have the following.

\begin{cor}\label{m2n2}
If $T$ is a tree of order $n$ but not a star, then $m_T(\mu_2)\leq \frac{n-3}{2}$.
\end{cor}

The multiplicity of $\mu_2$ in a star $K_{1,n}$ equals $n-1$ and by Corollary \ref{m2n2}, $m_T(\mu_2)\leq \lfloor \frac{n-1}{2} \rfloor -1$ for every other tree $T$. So there is a huge gap between these multiplicities $n-1$ and $ \lfloor \frac{n-1}{2} \rfloor -1$.
Lemma \ref{lemmatsk35} and Proposition \ref{spectrum2} show that there are no other gaps. In fact, for any $i=1,\dots, \lfloor \frac{n-1}{2} \rfloor -1$, the tree $T=T(n-i-2,i+1)$ has $m_T(\mu_2)=i$.

In Theorem \ref{thm:n/2-k}, we will show that for every positive integer $j$, all except finitely many trees $T$ with $m_T(\mu_2)=\lfloor \frac{n}{2} \rfloor -j$ are spiders. But first, we determine the typical values of the Laplacian eigenvalues with large multiplicity.

\begin{lem}\label{lem:n/2-k}
Let $j$ be an integer with $j \geq 2$. If $T$ is a tree of order $n> 10j$ and $m_T(\mu)= \lfloor \frac{n}{2} \rfloor -j$ for a Laplacian eigenvalue $\mu \neq 1$, then $\mu =\frac{3 \pm \sqrt{5}}{2}$.
\end{lem}

\begin{proof}
Suppose that $n\geq 10j+1$ and $m_T(\mu)=\lfloor \frac{n}{2} \rfloor -j$. Because $\mu$ cannot be integer, there exists another (conjugate) eigenvalue $\overline{\mu}$ with $m_T(\overline{\mu})=\lfloor \frac{n}{2}\rfloor -j$. Since $n\geq 10j+1$ there can be no other (conjugate) eigenvalue with such a large multiplicity. Thus, $\mu$ and $\overline{\mu}$ are the roots of a quadratic factor of the Laplacian polynomial, so $\mu + \overline{\mu}$ equals a positive integer $t$, say. Because $\mu\overline{\mu}=1$ by Lemma \ref{van}, it follows that $t \geq 3$. Since $\mu_1+\cdots +\mu_n =2n-2$, we have that
$(\lfloor \frac{n}{2} \rfloor -j)t < 2n-2.$

First suppose  that $t\geq 5$. If $n$ is even, then it follows that $n<\frac{2jt-4}{t-4} \leq 10j-4$, which is a contradiction. Similarly, if $n$ is odd, then $n<\frac{(2j+1)t-4}{t-4} \leq 10j+1$; again a contradiction.
Next, suppose that $t=4$. Note that by Proposition \ref{delta} and Lemma \ref{haem}, we have $\Delta \geq \lfloor \frac{n}{2} \rfloor -j +1$ and $\mu_n \geq \Delta +1$, and hence  $\mu_n \geq \lfloor \frac{n}{2}\rfloor -j+2$. Note also that $\mu_n$ is a simple eigenvalue by \cite[Prop.~2.2]{kamel}, so it is not equal to $\mu$ or $\overline{\mu}$. If $n$ is even, then
$$ \sum_{\mu_i \notin \{ \mu, \overline{\mu} \}} \mu_i = 4j-2 >  \frac{n}{2} -j+2$$
and so $n<10j-8$, which is a contradiction.
Similarly, if $n$ is odd, then
$$ \sum_{\mu_i \notin \{ \mu, \overline{\mu} \}} \mu_i = 4j > \frac{n-1}{2} -j+2$$
and so $n<10j-3$; again a contradiction.

Thus, $t= 3$ and hence $\mu ,\overline{\mu} = \frac{3 \pm \sqrt{5}}{2}$.
\end{proof}

\begin{thm}\label{thm:n/2-k}
Let $j$ be an integer with $j\geq 2$ and let $T$ be a tree of order $n> 10j$. Then $m_T(\mu_2)= \lfloor \frac{n}{2} \rfloor -j$ if and only if $T \cong T(\lceil \frac{n}{2} \rceil +j-2, \lfloor \frac{n}{2} \rfloor -j+1)$.
\end{thm}

\begin{proof}
If $m_T(\mu_2)= \lfloor \frac{n}{2} \rfloor -j$, then $T$ is not a star, so $\mu_2 \neq 1$, and hence by Lemma \ref{lem:n/2-k}, we have $\mu_2= \frac{3-\sqrt{5}}{2}$. Now, by Proposition \ref{kho2}, it is clear that $T\cong T(s,k)$ for some suitable $s$ and $k$. These integers are now determined by the multiplicity of $\mu_2$ (see Propositions \ref{spectrum1} and \ref{spectrum2}), which finishes the proof.
\end{proof}

Observe the contrast between Theorems \ref{km2} and \ref{thm:n/2-k} in the sense that in the latter we restrict to eigenvalue $\mu_2$. Indeed, the following examples show that this restriction is necessary, at least for $j \geq 3$.

\begin{figure}[h]
\centering
\includegraphics[width=70mm]{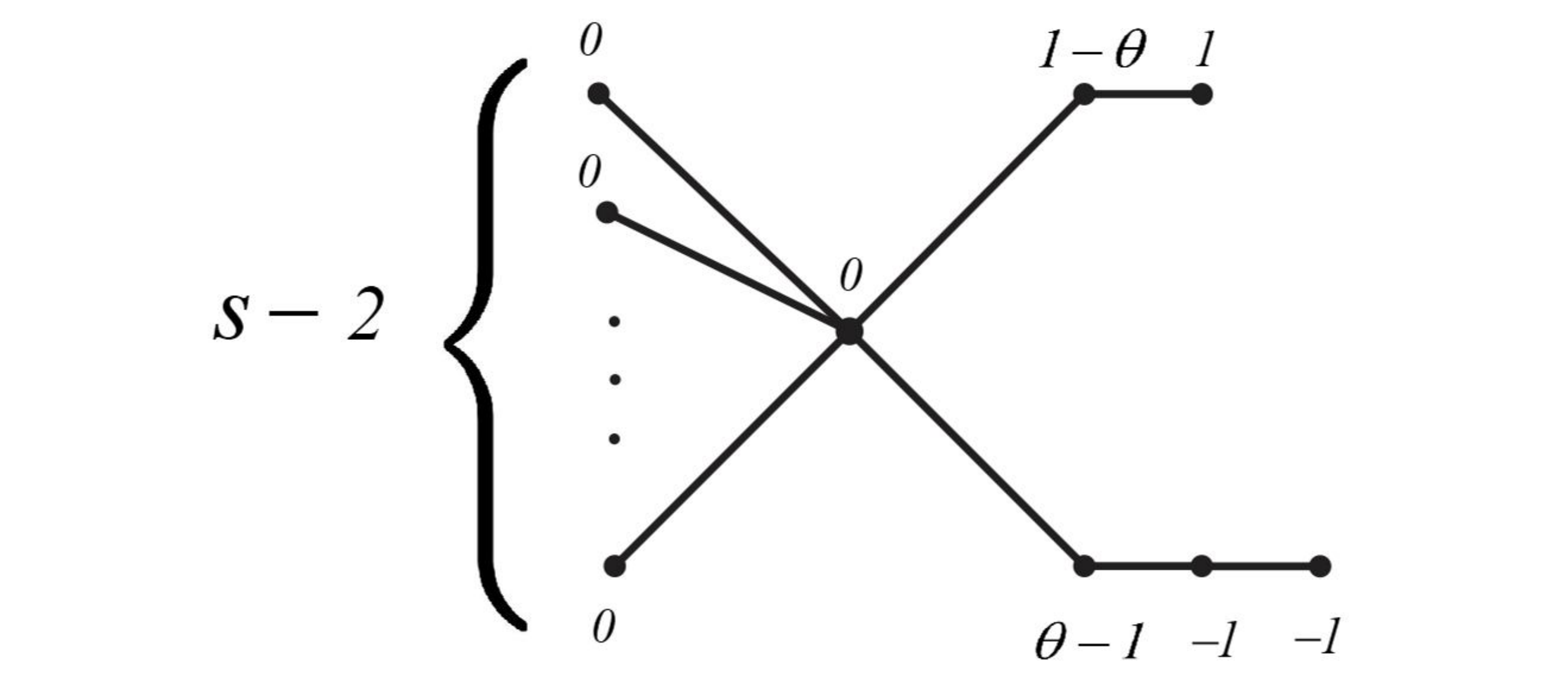}
\caption{The tree $T^{*}(s,2)$}\label{teta}
\end{figure}

Let $s \geq k \geq 2$ and let $T^{*}=T^{*}(s,k)$ be the tree obtained from $T(s,k)$ by adding one pendant vertex to one of the $k$ legs. See Figure \ref{teta} for the case $k=2$. It is clear from Proposition \ref{kho2} that $\mu_2(T^{*})< \eva$, and it follows from interlacing (Lemma \ref{kame22}) that $m_{T^{*}}(\eva)$ equals $k-2$ or $k-1$. We claim that this multiplicity equals $k-2$. In order to show this, we first consider the case $k=2$ and show that $\eva$ is not an eigenvalue of $T^{*}(s,2)$ for any $s \geq 2$.

Indeed, suppose that $\eva$ is an eigenvalue of $T^{*}(s,2)$ with eigenvector $x$. By normalizing the entry $x_1=1$ of the top right vertex in Figure \ref{teta}, and applying the equation $Lx=\eva x$, we recursively find the entries of $x$ as given in the figure (we omit the technical details; note also that it easily follows that $x_1$ cannot be zero). But then finally for the bottom two right vertices (where we obtained $x_{n-1}=x_n=-1$) we should have the equation $x_n-x_{n-1}=\eva x_{n}$, which gives a contradiction.

Next, consider the general case $T^{*}(s,k)$, and suppose that it has eigenvalue $\eva$ with multiplicity $k-1$, for $k\geq 3$. By interlacing it then follows that $T^{*}(s,k-1)$ has eigenvalue $\eva$ with multiplicity at least $k-2$, and by repeating this, we find that $T^{*}(s,2)$ has eigenvalue $\eva$ with multiplicity at least $1$, which is a contradiction, and hence confirms our claim.

For even $n$, we can now take $s=\frac n2+j-4$ for $j \geq 3$, to obtain $m_{T^{*}}(\eva)=\frac n2-j$. For odd $n$, we take  $s=\frac {n-1}2+j-3$ for $j \geq 3$, to obtain $m_{T^{*}}(\eva)=\frac {n-1}2-j$. Thus, for every $j \geq 3$ there are infinitely many trees with such multiplicities.

In the given examples, we cannot take $j=2$. Indeed, for this case we can prove something stronger than in Theorem \ref{thm:n/2-k}, and obtain a result that is similar to Theorem \ref{km2}. Here we will show that there are unique trees of order $n$ with multiplicities $\frac{n}{2}-2$ and $\frac{n-1}{2}-2$ for an eigenvalue $\mu\neq 1$.

\begin{thm}\label{n/2-2}
Let $T$ be a tree of order $n\geq 8$ and $\mu \neq 1$ be a Laplacian eigenvalue of $T$. Then $m_T(\mu)= \frac{n}{2} -2$ if and only if $T\cong T(\frac{n}{2}, \frac{n}{2} -1)$ and $\mu=\frac{3\pm \sqrt{5}}{2}$.
In particular, in this case $\mu$ is the algebraic connectivity of $T$ or its conjugate.
\end{thm}

\begin{proof} One side of the equivalence is clear from Proposition \ref{spectrum2}. To show the other side, we will first show by induction on $n$ (with $n$ even) the claim that if $m_T(\mu)= \frac{n}{2} -2$, then $\mu=\frac{3\pm \sqrt{5}}{2}$ and that  $\mu_2(T)=\eva=\frac{3- \sqrt{5}}{2}$. We checked that this is true for $n=8$ by enumerating all trees of this order with the Sage computer package.
For $n>8$, consider a tree $T$ with $m_T(\mu)= \frac{n}{2} -2$. Consider one of its extremal quasipendant vertices $v$ (in the sense that $v$ becomes pendant in the tree that is obtained from $T$ by removing all its pendant vertices). Now, remove the edge that connects $v$ to its (unique) non-pendant neighbor. The remaining graph is a disjoint union of a star (with center $v$) and a tree $T'$ on $t$ vertices, say. By interlacing (Lemma \ref{kame3}), it follows that $m_{T'}(\mu) \geq \frac n2-3$. But $m_{T'}(\mu) \leq \frac {t-3}2$ by Theorem \ref{km2}, so it follows that $t \geq n-3$. Because $n-t \geq 2$, we have two cases. If $t=n-3$, then $m_{T'}(\mu) = \frac {t-3}2$, and it follows from Theorem \ref{km2} that indeed $\mu=\frac{3\pm \sqrt{5}}{2}$ and $\mu_2(T')=\eva$, which implies by interlacing that $\mu_2(T)=\eva$. If $t=n-2$, then $m_{T'}(\mu) = \frac {t}2-2$, so by induction $\mu=\frac{3\pm \sqrt{5}}{2}$ and $\mu_2(T')=\eva$, but the latter (and interlacing) again implies that $\mu_2(T)=\eva$, which finishes the proof of our claim.

Just like in the proof of Theorem \ref{thm:n/2-k}, we can now apply Proposition \ref{kho2} to finish the proof.
\end{proof}

We omit the proof of the following similar result, as the proof is very similar, although we also have to use Theorem \ref{n/2-2} now.

\begin{thm}\label{n-1/2-2}
Let $T$ be a tree of order $n\geq 9$ and $\mu \neq 1$ be a Laplacian eigenvalue of $T$. Then $m_T(\mu)= \frac{n-1}{2} -2$ if and only if $T\cong T(\frac{n-1}{2}+1, \frac{n-1}{2} -1)$ and $\mu=\frac{3\pm \sqrt{5}}{2}$.
In particular, in this case $\mu$ is the algebraic connectivity of $T$ or its conjugate.
\end{thm}

We note that one could try to extend this result further, and indeed, the induction steps work over and over again, so induction would give a theorem about multiplicity $\lfloor \frac{n}{2} \rfloor -j$ for each $j \geq 2$, if only the bases of the induction steps would be true. Of course, this is where things go wrong. For example, we cannot prove a similar result for multiplicity $\frac{n}{2} -3$ and $n \geq 10$ because we have counterexamples such as $T^{*}(4,4)$ on $10$ vertices. Indeed, there are five such counterexamples, and we depict them in Figure \ref{5examples2}. Besides $T^{*}(4,4)$ (Fig.~\ref{5examples2}a), there are (exactly) two other trees of order $10$ that have eigenvalue $\eva$ with multiplicity $2$ and $\mu_2(T)<\eva$ (Fig.~\ref{5examples2}b and c). Moreover, there is one tree of order $10$ with eigenvalues $2 \pm \sqrt{3}$ with multiplicity $2$ (Fig.~\ref{5examples2}d), and one tree of order $10$ that has the roots of $x^3-5x^2+6x-1$ as eigenvalues with multiplicity $2$ (Fig.~\ref{5examples2}e). The variety of these examples shows that further classification of trees with an eigenvalue other than $1$ with multiplicity $\lfloor \frac{n}{2} \rfloor -j$ for $j \geq 3$ (on top of Lemma \ref{lem:n/2-k} and Theorem \ref{thm:n/2-k}) seems unfeasible.

\begin{figure}[h]
\centering
\includegraphics[width=110mm]{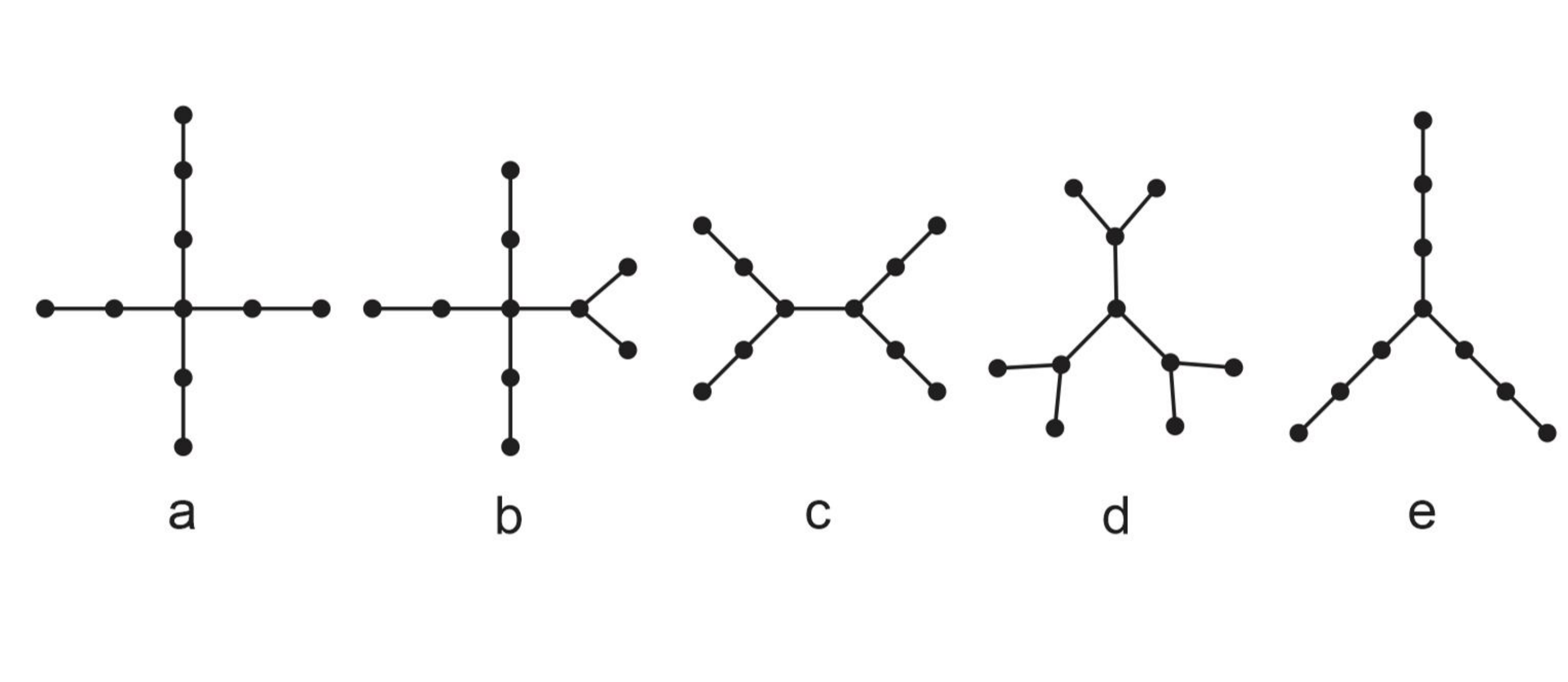}
\caption{Five counterexamples of order $10$}\label{5examples2}
\end{figure}

\noindent  {\bf Acknowledgement}. The research of the first author was partly funded by Iran National Science Foundation (INSF) under the contract No. 96004167. Also, part of the work in this paper was done while the third author was visiting Tilburg University.



\end{document}